\documentclass[11pt]{amsart}

\def \R{\mathbb{R}}

\newtheorem{theorem}{Theorem}[section]

\newtheorem{cor}[theorem]{Corollary}
\newtheorem{defi}[theorem]{Definition}
\newtheorem{lema}[theorem]{Lemma}

\newtheorem{rem}[theorem]{Remark}

\newtheorem{ejem}[theorem]{Example}

\newtheorem{prop}[theorem]{Proposition}

\begin{document}

%-------------------------------------------------------
\title[Surfaces with constant principal angles with a plane]
{Surfaces in $\R^4$ with constant principal angles with respect to a plane}

\author[P. Bayard]{Pierre Bayard}
\author[A. J. Di Scala]{Antonio J. Di Scala}
\author[O. Osuna Castro]{Osvaldo Osuna Castro}
\author[G. Ruiz-Hern\'andez]{Gabriel Ruiz-Hern\'andez}

\thanks{The fourth author was partially supported by Conacyt.}

\date{}

\maketitle

\begin{abstract}
We study surfaces in $\R^4$ whose tangent spaces have constant principal angles with respect to a plane. Using a PDE we prove the existence of surfaces with arbitrary constant principal angles. The existence of such surfaces turns out to be equivalent to the existence of a special local symplectomorphism of $\R^2$. We classify all surfaces with one principal angle equal to $0$ and observe that they can be constructed as the union of normal holonomy tubes. We also classify the complete constant angles surfaces in $\R^4$ with respect to a plane. They turn out to be extrinsic products. We characterize which surfaces with constant principal angles are compositions in the sense of Dajczer-Do Carmo. Finally, we classify surfaces with constant principal angles contained in a sphere and those with parallel mean curvature vector field.
\end{abstract}

\section{Introduction}

In \cite{Jordan} Camille Jordan defined the concept of principal angles between two linear subspaces of the Euclidean space. The principal angles are real numbers between $0$ and $\frac{\pi}{2}$ which describe the mutual position of the two subspaces. If one subspace has dimension one, the principal angle is just the usual angle between a straight line and a subspace. When both subspaces have dimension two, their principal angles are two real numbers $\theta_1,\theta_2$ such that $0\leq\theta_1 \leq \theta_2 \leq \frac{\pi}{2}.$ 

In this work we consider the principal angles between the tangent planes of an immersed surface in $\R^4$ and a fixed plane $\Pi$ in $\R^4$. Every tangent plane $T_p\Sigma$ of the surface $\Sigma$, considered as a vector subspace of $\R^4$, has two principal angles $\theta_1(p), \theta_2(p)$ with the fixed plane $\Pi,$ which depend on the point $p \in \Sigma$. The aim of this article is to investigate local and global geometric properties of those surfaces in which $\theta_1(p),\theta_2(p)$ are constant functions. For simplicity, we will call them helix surfaces or constant angles surfaces with respect to a plane. 

In the case where the constant angles surface is contained in some hyperplane $\R^3$ of $\R^4$, we show in Proposition \ref{helices-nonfull}, that the surface has a constant angle with respect to a direction in the hyperplane $\R^3;$ these surfaces are classified in \cite{DS-RH} and \cite{Mu-Ni}.
The case of constant angle submanifolds in $\R^n$ with respect to some direction in $\R^n$ was investigated by the second and the last authors in \cite{DS-RH-II}.
Constant angle surfaces with respect to a direction have been investigated very recently also in other
Riemannian manifolds, as in the works \cite{Di-Mu} and \cite{DFV}.

Here is a theorem collecting some of our main results.
\begin{theorem}\label{main} Let $\Sigma \subset \mathbb{R}^4$ be a surface with constant principal angles with respect to a plane $\Pi \subset \R^4$. Then $\Sigma$ has zero Gauss curvature and has flat normal bundle. If $\Sigma$ is complete then $\Sigma$ is an extrinsic product. Moreover, if $\Sigma$ is compact then $\Sigma$ is a torus embedded in $\mathbb{R}^4$ as a product of two closed plane curves.
\end{theorem}
In Section \ref{preliminaries} we recall basic definitions and properties of principal angles between linear subspaces in $\R^4,$ and we introduce the notion of helix (or constant angles) surfaces. In Section \ref{sec:gauss-map} we study the Gauss map of a helix surface; we show that a surface $\Sigma$ has constant principal angles if and only if its Gauss map image belongs to a product of circles in $S^2(\sqrt{2}/2)\times S^2(\sqrt{2}/2)$. If $\Sigma$ is moreover compact then the circles are equators in $S^2(\sqrt{2}/2)$.

In Section \ref{section structure equations} we write down the structure equations of a constant angles surface in an adapted frame, and in Section \ref{section complete surfaces} we classify the complete helix surfaces such that $0<\theta_1<\theta_2<\pi/2$ (\emph{the generic case}) and such that $0=\theta_1<\theta_2<\pi/2.$ In Section \ref{sec:compositions}, we characterize the constant angles surfaces which are compositions, a concept studied in \cite{DT1},\cite{DT2} and \cite{DoDa} by Do Carmo, Dajczer and Tojeiro in the context of local isometric immersions of $\mathbb{R}^2$ into $\mathbb{R}^4$ with zero normal curvature. We prove that a generic helix surface is a composition if and only if its first normal space has rank one (Proposition \ref{CNT}). 

In Section \ref{section surfaces spheres} we study constant angles surfaces in spheres of $\R^4$  and in Section \ref{section structure} we describe the local structure of constant angles surfaces whose lower principal angle vanishes; we show that these surfaces can be constructed as the union of holonomy tubes along a curve in the normal space of a given curve of $\R^4.$ A similar construction was used in \cite{DO}.

In Section \ref{Ex} and \ref{section surfaces without geodesic} we show the existence of surfaces with constant principal angles; we use the Cauchy-Kowalewski existence theorem for partial differential equations. In Theorem \ref{thm:existence-revisited}, we show the existence of non trivial helix surfaces with generic principal angles and whose first normal spaces have rank two; in particular these helix surfaces are not compositions. In order to use the Cauchy-Kowalewski theorem we consider the surface $\Sigma$ as the graph of a local diffeomorphism $F: U \subset \mathbb{R}^2 \rightarrow \mathbb{R}^2$. Then we observe that $\Sigma$ is a helix surface if and only if $F$ is a symplectomorphism whose jacobian matrix has constant length. It is interesting to remark that, by Theorem \ref{main}, a global symplectomorphism $F:\mathbb{R}^2 \rightarrow \mathbb{R}^2$ whose Jacobian matrix has constant length is necessarily an affine map.

In the last section, we consider constant angle surfaces with parallel mean curvature vector. We prove in Theorem \ref{thm:parallel-meancurvature}, that the latter condition is equivalent to the fact that $\Sigma$ is a product.

\section{Preliminaries}\label{preliminaries}

\subsection{Principal angles.}\label{section principal angles}

Here we recall the notion of principal angles between two planes in $\R^4.$ We refer to \cite{Jordan} or \cite{Jiang} for more details.

\begin{defi}
\label{dfn:angulos-principales}
\em
Let $V$ and $W$ be two-dimensional subspaces of $\R^4$. The principal angles between $V$ and $W$, $0 \leq \theta_1 \leq \theta_2 \leq \pi /2$, are defined by
$$\cos \theta_1 := \langle v_1, w_1 \rangle :=
 \max \{ \langle v , w \rangle | v \in V, w \in W, |v|=|w|=1 \},$$
$$\cos \theta_2 := \langle v_2 , w_2 \rangle :=
\max \{ \langle v , w \rangle  | v \in V, w \in W, v \perp v_1, w \perp w_1, |v|=|w|=1 \} .$$
\end{defi}
If $p_{W}:\R^4\rightarrow W$ stands for the orthogonal projection on $W,$ the expression $\mathrm{Q}_{WV}(v) := \langle \mathrm{p}_{W}(v)  , \mathrm{p}_{W}(v) \rangle$ defines a quadratic (positive semidefinite) form on the subspace ${V}$. Let us denote by $\mathcal{S}_{WV} \in Sym(V)$ the symmetric endomorphism such that
\[ \langle \mathrm{p}_{W}(v)  , \mathrm{p}_{W}(v') \rangle = \langle \mathcal{S}_{WV}(v),v' \rangle, \, \, \, \forall v,v' \in V. \]
The following is well-known.
\begin{prop}\label{Lagrange} 
The eigenvalues of $\mathcal{S}_{WV}$ are $\cos^2(\theta_1)$ and $\cos^2(\theta_2).$ In particular, there exists an orthonormal basis $(v_1,v_2)$ of $V$ such that, for all $v=X_1v_1+X_2v_2$ belonging to $V,$
\begin{equation}\label{QWV}
Q_{WV}(v)=\cos^2\theta_1 X_1^2+\cos^2\theta_2X_2^2.
\end{equation}
\end{prop}
The next lemma links the principal angles between $V$ and $W^{\perp}$ to the principal angles between $V$ and $W:$
\begin{lema}
\label{principal-angles-orthogonal}
Let $V$ and $W$ be two-dimensional subspaces of $\R^4$. If the principal angles between $V$ and $W$ are $\theta_1$ and $\theta_2$, then the principal angles  between $V$ and $W^\perp$ are $\theta_1^\perp =\pi /2 - \theta_2 \leq \theta_2^\perp =\pi /2 - \theta_1$.
\end{lema}
\begin{proof}
Since, for all $v\in V,$  $v=p_W(v)+p_{W^{\perp}}(v)$ with $p_W(v)\perp p_{W^{\perp}}(v),$ we readily get
$$|v|^2=Q_{WV}(v)+Q_{W^{\perp}V}(v).$$
Thus, by (\ref{QWV}),
$$Q_{W^{\perp}V}(v)=\sin^2\theta_1 X_1^2+\sin^2\theta_2X_2^2.$$
Using Proposition \ref{Lagrange} again (with $W^{\perp}$ instead of $W$), we deduce that $\cos^2(\theta_1^\perp)=\sin^2\theta_2$ and $\cos^2(\theta_2^\perp)=\sin^2\theta_1,$ and the result follows.
\end{proof}
\begin{prop}
Given any two angles $0 \leq \theta_1 \leq \theta_2 \leq \pi /2$, there exist two planes $V$ and $W$ in $\R^4$ with these two principal angles.
\end{prop}
\begin{proof}
Let $(w_1,w_2,w_3,w_4)$ be an orthonormal basis of $\R^4,$ and consider the plane
$W= span\{w_1, \ w_2\}$. Let $v_1$ and $v_2$ be the orthonormal vectors given by
$$v_1:= \cos(\theta_1) w_2 + \sin(\theta_1) w_4 \ \ \ v_2:= \cos(\theta_2) w_1 + \sin(\theta_2) w_3 .$$
Let us define $V=span\{v_1, \ v_2\}$. Since $p_W(v_1)=\cos(\theta_1) w_2$ and $p_W(v_2)=\cos(\theta_2) w_1,$ writing $v=X_1v_1+X_2v_2$ we readily get
$$Q_{WV}(v)=|p_W(v)|^2=\cos^2(\theta_1)X_1^2+\cos^2(\theta_2)X_2^2,$$
and the result follows from (\ref{QWV}).
\end{proof}

\subsection{Principal angles and bivectors}\label{sectionanglesbivectors}
We consider the vector space $\Lambda^2\R^4$ endowed with its natural scalar product, defined on decomposable bivectors by
$$\langle v_1\wedge v_2,w_1\wedge w_2\rangle:=\langle v_1,w_1\rangle\langle v_2,w_2\rangle-\langle v_2,w_1\rangle\langle v_1,w_2\rangle.$$
Let $V$ and $W$ be two oriented planes of $\R^4.$ If $(v_1,v_2)$ and $(w_1,w_2)$ are positively oriented and orthonormal basis of $V$ and $W$, we define the angle $\theta\in [0,\pi]$ between $V$ and $W$ by  the formula
$$\cos\theta=\langle v_1\wedge v_2,w_1\wedge w_2\rangle.$$
Let us denote by  $\theta^{\perp}\in[0,\pi]$ the angle between $V$ and $W^{\perp}$, where the orientation of $W^\perp$ is such that the union of two positively oriented basis of $W$ and $W^\perp$ is a positively oriented basis of $\R^4.$ Note that $\theta^{\perp}$ is also the angle between $V^{\perp}$ and $W.$ The following result may be find in \cite{Jiang}, Theorem 5.
\begin{lema}\label{lemaformulatheta}
The angles $\theta$ and $\theta^{\perp}$ are linked to the principal angles $\theta_1$ and $\theta_2$ between $V$ and $W$ by the formulae
\begin{equation}\label{formulatheta1}
|\cos\theta|=\cos\theta_1.\cos\theta_2\hspace{1cm}\mbox{and}\hspace{1cm}
|\cos\theta^{\perp}|=\sin\theta_1.\sin\theta_2.
\end{equation}
\end{lema}

\subsection{Surfaces with constant principal angles.}
Recall that a surface $\Sigma \subset \mathbb{R}^4$ is called \emph{full} if it is not contained in an affine hyperplane.
\begin{defi}
\label{dfn:helix-surface}
\em
Let $\Sigma$ be an immersed surface in $\R^4$ and let $\Pi \subset \R^4$ be a two-dimensional plane.
We say that $\Sigma$ is a {\em helix surface} or a {\em constant angles surface} with respect to $\Pi$, if the principal angles between $T_pM$ and $\Pi$ do not depend on $p \in \Sigma$. We will also say that $\Sigma$ has \emph{constant principal angles} with respect to the plane $\Pi$.
\end{defi}
\begin{ejem}{(The Clifford torus is a helix surface)\\}
\em
Let us consider the torus $T^2= \mathbb{S}^1 \times \mathbb{S}^1 \subset \R^2 \times \R^2=\R^4$ with the metric induced by the metric of $\R^4.$ Let us see that $T^2$ is a helix surface with respect to the plane $\Pi_{12}=\{(x_1, x_2, x_3, x_4) \in \R^4 | x_3=x_4=0 \}$. The tangent space of $T^2$ at the point $p=(x_1, x_2, x_3, x_4)\in T^2$ has an orthonormal basis given by $(v_1=(-x_2, x_1,0,0)$, $v_2=(0,0,-x_4,x_3))$. Since $w_1 := v_1$ belongs to $T_pT^2 \cap \Pi_{12}$, we readily get $\cos \theta_1=\langle v_1, w_1 \rangle = 1$, i.e. $\theta_1=0$. Moreover, we get that $v_2 \perp v_1,$ $w_2:=(x_1, x_2,0,0)$ belongs to $\Pi_{12}$ and that $w_2 \perp w_1$. Thus $\cos \theta_2=\langle v_2, w_2 \rangle = 0$, i.e. $\theta_2 = \pi/2 $. So, the flat torus has constant principal angles $\theta_1=0, \theta_2=\pi/2$ with respect to the plane $\Pi_{12}$, i.e. $T^2$ is a helix with respect to $\Pi_{12}$. Analogously $T^2$ is a helix with respect to the plane $\Pi_{34} = \{(x_1, x_2, x_3, x_4) \in \R^4 | x_1=x_2=0 \}$, with the same constant principal angles.
\end{ejem}
\begin{ejem}
\em
We construct a helix surface in $\R^4$ with respect to a plane by using a nonplanar curve in $\R^3$.
The surface will be full and will be a Riemannian product of $\R$ with a curve in $\R^3$. Let $\gamma$ be a classical regular helix curve in $\R^3$ with respect to a fixed direction $d$, i.e.
such that the tangent vectors of $\gamma$ make a constant angle $\theta$ with $d$. We define $\Sigma$ as the Riemannian product $\gamma \times \R$ which is an immersed surface in $\R^3 \times \R = \R^4$. Then $\Sigma$
has constant principal angles with respect to the plane generated by $d$ and $e_4=(0,0,0,1)$. The
constant principal angles are $\theta$ and $0$.
\end{ejem}

\begin{lema}
\label{compactas-tienen-angulo-cero}
Let $\Sigma^2 \subset \R^4$ be a compact immersed surface. Let $\Pi$ be
any two-dimensional plane in $\R^4$. Then there exists $p \in \Sigma$, depending on $\Pi$,
such that $T_p\Sigma$ and $\Pi$ have a principal angle equal to zero.
\end{lema}
\begin{proof}
Let $H$ be any hyperplane containing the plane $\Pi.$ Since $\Sigma$ is compact there exists $p \in \Sigma$ such that $T_p\Sigma \subset H$. So, the two planes $T_p\Sigma$ and $\Pi$ belong to the hyperplane $H$. Therefore, they have a common straight line and thus a principal angle has to be zero.
\end{proof}
\begin{prop}
\label{mejora-enel-caso-compacto}
If $\Sigma$ is a compact immersed helix surface in $\R^4$ with respect to a plane,
then it has constant principal angles equal to zero and $\pi/2.$
\end{prop}
\begin{proof}
By Lemma \ref{compactas-tienen-angulo-cero}, there exists $p \in \Sigma$ with a principal angle at $p$ equal to zero. Because $\Sigma$ is a helix, $\Sigma$ has a zero principal angle at every point. Now, the same argument applied to $\Pi^{\perp}$ shows that the other principal angle is equal to $\pi/2$ (using also Lemma \ref{principal-angles-orthogonal}).
\end{proof}

\begin{ejem}
\label{helixsurfaceinR3-induce-surfaceinR4}
\em
We construct a noncompact helix surface in $\R^4$ with respect to a plane, with one principal angle equal to zero. In this example the helix surface is not full.
Let $\Sigma$ be an immersed surface in the Euclidean space $\R^3$ with its standard Riemannian
metric. Let us assume that there is a unit vector $d \in \R^3$ such that every tangent space
of $\Sigma$ makes a constant angle $0< \theta < \pi/2$ with the direction $d.$ The authors Di Scala and Ruiz-Hern\'andez investigated this class of submanifolds in \cite{DS-RH} and \cite{DS-RH-II}; they are called helix surfaces with respect to the direction $d,$ or constant angle surfaces. They are never compact. For example a cone of revolution is a helix with respect to a direction of its axis of revolution. Now, let us consider $\Sigma$ as an immersed surface in $\R^4,$ using the inclusion $\R^3\subset \R^3 \times \R= \R^4$. Let $\Pi$ be the plane generated by $d$ and $e_4=(0,0,0,1) \in \R^4$. Then $\Sigma$ is a helix surface with respect to $\Pi$ and its constant principal angles are $\pi/2$ and $\theta$. Let $\Pi^\perp$ be the orthogonal complement of $\Pi$ in $\R^4$. By Lemma \ref{principal-angles-orthogonal}, $\Sigma$ is a helix surface with respect to $\Pi^\perp$ and its constant principal angles are $\pi/2 - \theta$ and $0$.
\end{ejem}

There is a natural relation between non full helix surfaces in $\R^4$ with respect to a plane and helix surfaces in $\R^3$ with respect to a direction:

\begin{prop}{\emph{Classification of helix surfaces in $\R^4$ which are not full}.\\}
\label{helices-nonfull}
Let $\Sigma$ be a non full immersed surface in $\R^4$ which is a helix with respect to a plane $\Pi \subset \R^4$. Assume that $\Sigma$ is contained in $\R^3\subset\R^4.$ Then $\Sigma$ is a helix surface in $\R^3$ with respect to some direction in $\R^3$.
\end{prop}
\begin{proof}
\emph{First case:} $\Pi$ is contained in the hyperplane $\R^3$. Let $d$ be a unit vector in $\R^3$ normal to $\Pi$. Then $\Sigma$ is a helix surface with respect to the direction $d$. 
\emph{Second case:} $\Pi$ is transversal to $\R^3$. Thus $\Pi \cap \R^3$ is a line $l$ in $\R^3$. Let us denote by $d$ a fixed unit direction in $\R^3$ parallel to $l$. We prove that $\Sigma$ is a helix surface with respect to the direction $d$. For this, consider $e_4$ a unit vector normal to $\R^3$ in $\R^4$ and $\xi$ a local unit vector field orthogonal to $\Sigma$ in $\R^3$. Since $(\xi,e_4)$ is an orthonormal basis of $T\Sigma^{\perp},$ the bivector $\xi\wedge e_4$ represents the normal plane $T\Sigma^{\perp}.$ By hypothesis, for every $p \in \Sigma$, $T_p\Sigma$ and $\Pi$ have constant principal angles, and by Lemma \ref{principal-angles-orthogonal} $T_pM^\perp$ and $\Pi$ also have constant principal angles. Let $w \in \R^4$ be a fixed direction such that $(d,w)$ is an orthonormal basis of $\Pi$. We conclude that $\langle \xi \wedge e_4 , d \wedge w \rangle :=\langle \xi, d \rangle \langle e_4, w \rangle-\langle \xi, w \rangle \langle e_4,d \rangle =\langle \xi, d \rangle \langle e_4, w \rangle$ is constant (see Lemma \ref{lemaformulatheta}). Taking the derivative along a direction $T$ tangent to $\Sigma$ we get
\begin{eqnarray*}
0= T.(\langle \xi \wedge e_4 , d \wedge w \rangle)&=&(T.\langle \xi, d \rangle) \langle e_4, w \rangle+
\langle \xi, d \rangle (T .\langle e_4, w \rangle)\\
&=&(T.\langle \xi, d \rangle) \langle e_4, w \rangle
\end{eqnarray*} 
since $\langle e_4, w \rangle$ is constant ($e_4$ and $w$ are fixed directions). Finally, let us observe that $ \langle e_4, w \rangle \neq 0:$ otherwise $w$ would be in $\R^3,$ which is not possible since $\Pi,$ transversal to $\R^3,$ is generated by the pair $(d, \ w),$ with $d$ belonging to $\R^3.$ Therefore $T.\langle \xi, d \rangle =0,$ which means that $\langle \xi, d \rangle$ is constant along $\Sigma$. This is equivalent to say that $\Sigma$ is a helix surface with respect to the direction $d$.
\end{proof}

\begin{ejem}\emph{A helix in $\R^4$ which is full and is not a Riemannian product of two curves.} Let $\Pi$ be a two-dimensional subspace of $\R^4$ and let $G$ be the group of all isometries of $\R^4$ that fix pointwise $\Pi$. So, $G$ is isomorphic to the group $SO(2)$. Let $\gamma$ be a connected regular curve in $\R^4$, whose tangent lines make a constant angle with the plane $\Pi$. We define an immersed surface $\Sigma$ in $\R^4$ by taking $\Sigma:=G \cdot \gamma$, the orbit of $\gamma$ under the action of $G$. Let us observe that $\Sigma$ is foliated by its geodesics $g \cdot \gamma$, for every $g \in G$. The other curves $G \cdot p$ for every $p \in  \gamma$ (these curves are planar circles in $\R^4$) are orthogonal to such family of geodesics in $\Sigma$. Let us observe that the geodesics on $\Sigma$ given by $g \cdot \gamma$ have the same property as the original $\gamma$: their tangent lines make the same constant angle with respect to the plane $\Pi$, since $G$ consists of isometries in $\R^4$ that fix pointwise $\Pi$.
\end{ejem}

\section{Characterization of helix surfaces using the Gauss map}
\label{sec:gauss-map}
The Grassmannian of the oriented 2-planes in $\R^4$ identifies with the set
$$Q=\{\eta\in\Lambda^2\R^4:\ \langle\eta,\eta\rangle=1,\ \eta\wedge\eta=0\}$$
of unit and decomposable bivectors of $\R^4.$ Recall that the Hodge operator is the symmetric map $*:\Lambda^2\R^4\rightarrow\Lambda^2\R^4$ such that $\langle\eta,*\eta'\rangle=\eta\wedge\eta'$ for all $\eta,\eta'\in\Lambda^2\R^4,$ where $\Lambda^4\R^4$ is identified with $\R$ using the canonical volume form on $\R^4.$ Since $**=id_{\Lambda^2\R^4},$ $\Lambda^2\R^4$ splits into the orthogonal sum
$$\Lambda^2\R^4=E^+\oplus E^-$$
where $E^+=\{\eta:*\eta=\eta\}$ and $E^-=\{\eta:*\eta=-\eta\},$ and the natural map
\begin{eqnarray*}
\Lambda^2\R^4&\rightarrow&E^+\oplus E^{-}\\
\eta&\mapsto&(\eta^+,\eta^-),
\end{eqnarray*}
induces an isometry between $Q$ and the product of spheres $S^2(\sqrt{2}/2)\times S^2(\sqrt{2}/2).$ Consider the Gauss map
\begin{eqnarray*}
G:\Sigma&\rightarrow&S^2(\sqrt{2}/2)\times S^2(\sqrt{2}/2)\\
x&\mapsto& \left((e_1\wedge e_2)^+,(e_1\wedge e_2)^-\right)
\end{eqnarray*}
where $(e_1,e_2)$ is a positively oriented and orthonormal basis of $T_x\Sigma.$ We first give a characterization of an helix surface in terms of its Gauss map image:
\begin{prop}\label{proposition_Gaussmapimage}
$\Sigma$ is an immersed helix surface in $\R^4$ with respect to a plane if and only if its Gauss map image belongs to a product of circles in $S^2(\sqrt{2}/2)\times S^2(\sqrt{2}/2).$
\end{prop}
\begin{proof}
Fix $\Pi$ an oriented plane of $\R^4,$ represented by $(\eta_o^+,\eta_o^-)\in  S^2(\sqrt{2}/2)\times S^2(\sqrt{2}/2).$ For $x\in \Sigma,$ $G(x)=(\eta^+,\eta^-)$ represents the plane $T_x\Sigma,$ with its orientation. We define the two angles $\alpha^+,\alpha^-$ by the formulae
$$\cos(\alpha^+)=2\langle \eta_o^+,\eta^+\rangle,\hspace{1cm}\cos(\alpha^-)=2\langle \eta_o^-,\eta^-\rangle.$$
The formulae
$$\langle\eta_o,\eta\rangle=\langle\eta_o^+,\eta^+\rangle+\langle\eta_o^-,\eta^-\rangle,\hspace{1cm}\langle\eta_o,*\eta\rangle=\langle\eta_o^+,\eta^+\rangle-\langle\eta_o^-,\eta^-\rangle$$
read
$$\cos\theta=\frac{1}{2}(\cos\alpha^++\cos\alpha^-),\hspace{1cm}\cos\theta^{\perp}=\frac{1}{2}(\cos\alpha^+-\cos\alpha^-),$$
where $\theta$ and $\theta^{\perp}$ are the angles between $\Pi$ and $T_x\Sigma,$  and $\Pi$ and $T_x\Sigma^{\perp}$ defined Section \ref{sectionanglesbivectors}. We thus have
\begin{equation}\label{alpha-theta}
\cos\alpha^+=\cos\theta+\cos\theta^{\perp}\hspace{1cm}\mbox{and}\hspace{1cm}\cos\alpha^-=\cos\theta-\cos\theta^{\perp}.
\end{equation}
By Lemma \ref{lemaformulatheta} we deduce that the angles $\alpha^+$ and $\alpha^-$ are constant if and only if the principal angles $\theta_1,\theta_2$ are, and thus that $\Sigma$ is an helix surface with respect to $\Pi$ if and only if its Gauss map image belongs to a product of circles centered at $\eta_0^+$ and $\eta_0^-$ in $S^2(\sqrt{2}/2)\times S^2(\sqrt{2}/2).$
\end{proof}

\begin{prop}
If $\Sigma$ is a compact immersed helix surface in $\R^4$ with respect to a plane $\Pi$,
then it has constant principal angles equal to zero and $\pi/2.$ That means that its Gauss map image is a product of two equators in $S^2(\sqrt{2}/2)\times S^2(\sqrt{2}/2).$
\end{prop}
\begin{proof}
The first part was proved in Proposition \ref{mejora-enel-caso-compacto}. Lemma \ref{lemaformulatheta} and formulae (\ref{alpha-theta}) imply that the two angles $\alpha^+$ and $\alpha^-$ between $\Pi$ and $T_p\Sigma$ are equal to $\pi/2,$ and thus that the Gauss map image is a product of two equators in $S^2(\sqrt{2}/2)\times S^2(\sqrt{2}/2).$
\end{proof}

\section{Structure equations}\label{section structure equations}

In this section we compute the structure equations (see \cite[p.10]{BCO}) of a helix surface in a frame adapted to the helix structure.

Let $\Sigma \subset \mathbb{R}^4$ be a surface with constant principal angles with respect to the plane $\Pi \subset \mathbb{R}^4$.
Let $T_1,T_2 \in \Gamma(\mathrm{T}\Sigma)$ be a local frame such that $T_1(p)$ and $T_2(p)$ are unit eigenvectors of $S_{\Pi T_p\Sigma}$ at every point $p\in\Sigma,$ and let $e_1,e_2$ be the corresponding frame of $\Pi,$ defined by
\[ \begin{array}{lcl}
     e_1 & = & \cos(\theta_1)T_1 + \sin(\theta_1)\xi_1 \, ,\\
     e_2 & = & \cos(\theta_2)T_2 + \sin(\theta_2)\xi_2\\
   \end{array}
\]
where $\xi_1,\xi_2$ are normal vector fields ($\theta_1$ and $\theta_2$ still denote the constant principal angles). Note that $T_1$ and $T_2$ (and thus $e_1,$ $e_2,$ $\xi_1$ and $\xi_2$) do exist since the eigenvalues of $S_{\Pi T_p\Sigma}$ are constant. Let $X$ be a vector field of $\Sigma$. Taking derivatives in both hands we get
\[ \begin{array}{lcl}
     D_X e_1 & = & \cos(\theta_1)D_X T_1 + \sin(\theta_1)D_X\xi_1\\
             & = & \cos(\theta_1)(\nabla_X T_1 + \alpha(X,T_1) ) + \sin(\theta_1)(\nabla^{\perp}_X\xi_1 - A_{\xi_1}(X))\\
             & = &  \cos(\theta_1)\nabla_X T_1 - \sin(\theta_1)A_{\xi_1}(X) + \cos(\theta_1)\alpha(X,T_1) + \sin(\theta_1)\nabla^{\perp}_X\xi_1
   \end{array}
\]
and
\[ \begin{array}{lcl}
     D_X e_2 & = & \cos(\theta_2)D_X T_2 + \sin(\theta_2)D_X\xi_2\\
             & = & \cos(\theta_2)(\nabla_X T_2 + \alpha(X,T_2) ) + \sin(\theta_2)(\nabla^{\perp}_X\xi_2 - A_{\xi_2}(X))\\
             & = &  \cos(\theta_2)\nabla_X T_2 - \sin(\theta_2)A_{\xi_2}(X) + \cos(\theta_2)\alpha(X,T_2)+\sin(\theta_2)\nabla^{\perp}_X\xi_2.
   \end{array}
\]

We can regard $\Sigma \times \Pi \rightarrow \Sigma$ as a trivial bundle endowed with a flat connection $D$. Then, there exists a function $f:\Sigma \rightarrow \mathbb{R}$ such that
\begin{equation*}
D_X e_1= df(X)e_2 \hspace{1cm}\mbox{and}\hspace{1cm}D_X e_2 =  -df(X)e_1.
\end{equation*}
Then from the above equations we get
\[ \begin{array}{ccc}
\cos(\theta_2)df(X)T_2 + \sin(\theta_2)df(X)\xi_2 &=&  \cos(\theta_1)\nabla_X T_1 - \sin(\theta_1)A_{\xi_1}(X)\\
                                                  & &  +\cos(\theta_1)\alpha(X,T_1) + \sin(\theta_1)\nabla^{\perp}_X\xi_1 \, \, , \\
-\cos(\theta_1)df(X)T_1 - \sin(\theta_1)df(X)\xi_1 &=& \cos(\theta_2)\nabla_X T_2 - \sin(\theta_2)A_{\xi_2}(X)\\
                                                  & &+ \cos(\theta_2)\alpha(X,T_2)+\sin(\theta_2)\nabla^{\perp}_X\xi_2.
\end{array}
 \]
Taking the normal and the tangent components we get
\begin{equation} \label{tangente}
\begin{array}{rcl}
\cos(\theta_2)df(X)T_2  &=& \cos(\theta_1)\nabla_X T_1 - \sin(\theta_1)A_{\xi_1}(X)  \, , \\
-\cos(\theta_1)df(X)T_1 &=& \cos(\theta_2)\nabla_X T_2 - \sin(\theta_2)A_{\xi_2}(X)
   \end{array}
\end{equation}
\begin{equation} \label{normal}
\begin{array}{rcl}
 \sin(\theta_2)df(X)\xi_2 &=&  \cos(\theta_1)\alpha(X,T_1) + \sin(\theta_1)\nabla^{\perp}_X\xi_1 \, \, , \\
 - \sin(\theta_1)df(X)\xi_1 &=& \cos(\theta_2)\alpha(X,T_2) + \sin(\theta_2)\nabla^{\perp}_X\xi_2.
   \end{array}
\end{equation}

Here is the first consequence of the above equations.

\begin{lema}\label{flat}
The Levi-Civita connection and the normal connection are flat.
\end{lema}
\begin{proof} Indeed, from equations (\ref{normal}) it follows that 
\[ \alpha(T_1,T_2) = 0\hspace{1cm} \mbox{ and }\hspace{1cm}\alpha(T_1,T_1) \perp \alpha(T_2,T_2) \, .\]
Now the first claim follows from Gauss equation and the second from Ricci equation.
\end{proof}

In the frame $T_1,T_2$ we have \[ A_{\xi_1} = \left(
                      \begin{array}{cc}
                        0 & 0 \\
                        0 & m_1 \\
                      \end{array}
                    \right),\hspace{1cm}
A_{\xi_2} = \left(
                      \begin{array}{cc}
                        m_2 & 0 \\
                        0 & 0 \\
                      \end{array}
                    \right)
\] 
where $m_1,m_2:\Sigma\rightarrow\R$ are two smooth functions, and
\[ \alpha(T_1,T_1) = m_2 \xi_2 , \hspace{1cm} \alpha(T_2,T_2) = m_1 \xi_1.\]

Define $t,n:\Sigma\rightarrow\R$ two smooth functions such that
\begin{equation}\label{def t n}
\nabla T_1=dt\ T_2,\ \nabla T_2=-dt\ T_1,\ \nabla^{\perp} \xi_1=dn\ \xi_2\mbox{ and } \nabla^{\perp} \xi_2=-dn\ \xi_1.
\end{equation}
We may thus re-write the structure equations (\ref{tangente})-(\ref{normal}) as follows:
\begin{equation} \label{tangente1}
\begin{array}{rcr}
\cos(\theta_2)df(X)  &=&  \cos(\theta_1)dt(X) - \sin(\theta_1)\langle X,T_2\rangle m_1  \, , \\
-\cos(\theta_1)df(X) &=& -\cos(\theta_2)dt(X) - \sin(\theta_2)\langle X,T_1\rangle m_2 
   \end{array}
\end{equation}
and
\begin{equation} \label{normal1}
\begin{array}{rcl}
 \sin(\theta_2)df(X)&=&  \cos(\theta_1)\langle X,T_1\rangle m_2   + \sin(\theta_1)dn(X) \, \, , \\
 - \sin(\theta_1)df(X) &=& \cos(\theta_2)\langle X,T_2\rangle m_1  - \sin(\theta_2)dn(X).
   \end{array}
\end{equation}
Observe that equations (\ref{normal1}) imply the existence of two functions $\lambda_1,\lambda_2$ such that
\[ \nabla \lambda_1 = m_1 T_2\hspace{1cm} \mbox{ and }\hspace{1cm}\nabla \lambda_2 = m_2 T_1. \]

\begin{rem}
\em
The fact that $m_1 T_2$ and $m_2 T_1$ are gradients also follows from Codazzi equations (C1) and (C2) below.
\end{rem}

\subsection{The case $\theta_1 = 0$ and $ 0 < \theta_2 < \frac{\pi}{2}$}
\label{subsec:un-angulo-zero}

Under these assumptions equations (\ref{tangente1}) and (\ref{normal1}) are equivalent to
\begin{equation} \label{tangente0}
\begin{array}{rcl}
\cos(\theta_2)df(X)&=&  dt(X)\, , \\
-df(X)&=& -\cos(\theta_2)dt(X)- \sin(\theta_2)\langle X,T_1\rangle m_2
   \end{array}
\end{equation}
and
\begin{equation} \label{normal0}
\begin{array}{rcl}
 \sin(\theta_2)df(X)&=&  \langle X,T_1\rangle m_2 \, \, , \\
 0 &=& \cos(\theta_2)\langle X,T_2\rangle m_1 - \sin(\theta_2)dn(X),
   \end{array}
\end{equation}
and are thus equivalent to
\begin{equation}\label{zero}    d\lambda_1= \tan(\theta_2) dn,\hspace{.5cm}d\lambda_2 = \tan(\theta_2) dt\hspace{.5cm}\mbox{and}\hspace{.5cm}df =\frac{dt}{\cos(\theta_2)}.
\end{equation}
As a consequence we get the following result.
\begin{prop} Under the above assumptions the vector field $T_2$ is geodesic.
\end{prop}
\begin{proof} Indeed, from the second equation of the system (\ref{zero}) it follows that $dt(T_2) \equiv 0,$ which implies $\nabla_{T_2} T_2 = -dt(T_2) T_1 \equiv 0$.
\end{proof}

\subsection{The generic case}

By the \emph{generic case} we mean the case in which the principal angles $\theta_1,\theta_2 \notin \{0,\frac{\pi}{2}\}$. Under this hypothesis, we may eliminate $df$ in the first equation of (\ref{tangente1}) using successively the second, the third, and the last equation in (\ref{tangente1})-({\ref{normal1}), to get the system
\begin{eqnarray}
\left(\frac{\cos(\theta_1)}{\cos(\theta_2)}-\frac{\cos(\theta_2)}{\cos(\theta_1)}\right)dt &=&   \frac{\sin(\theta_1)}{\cos(\theta_2)}d\lambda_1 + \frac{\sin(\theta_2)}{\cos(\theta_1)}d\lambda_2\label{dependencia1}\\
    -\frac{\sin(\theta_1)}{\sin(\theta_2)}dn + \frac{\cos(\theta_1)}{\cos(\theta_2)} dt &=&   \frac{\sin(\theta_1)}{\cos(\theta_2)}d\lambda_1 + \frac{\cos(\theta_1)}{\sin(\theta_2)}d\lambda_2\label{dependencia2}\\
   -\frac{\sin(\theta_2)}{\sin(\theta_1)} dn + \frac{\cos(\theta_1)}{\cos(\theta_2)} dt &=& \left(\frac{\sin(\theta_1)}{\cos(\theta_2)}- \frac{\cos(\theta_2)}{\sin(\theta_1)}\right)d\lambda_1.\label{dependencia3}
\end{eqnarray}

\subsection{Codazzi equations}

Here we compute the Codazzi equations. Since

\[\begin{array}{lcl}
    (\nabla_{T_1} A_{\xi_1})(T_2) & = & \nabla_{T_1} \left(A_{\xi_1}(T_2)\right) - A_{\nabla^{\perp}_{T_1}\xi_1}(T_2) - A_{\xi_1}(\nabla_{T_1}T_2) \\
     & = & \nabla_{T_1} \left(A_{\xi_1}(T_2)\right) \\
     & = & \nabla_{T_1} (m_1 T_2) \\
     & = & dm_1(T_1) T_2 - m_1 dt(T_1) T_1, \\
  \end{array}
 \]
\[\begin{array}{lcl}
    (\nabla_{T_2} A_{\xi_1})(T_1) & = & \nabla_{T_2} \left(A_{\xi_1}(T_1)\right) - A_{\nabla^{\perp}_{T_2}\xi_1}(T_1) - A_{\xi_1}(\nabla_{T_2}T_1) \\
    & = &  - A_{\nabla^{\perp}_{T_2}\xi_1}(T_1) - A_{\xi_1}(\nabla_{T_2}T_1) \\
    & = &  - A_{dn(T_2)\xi_2}(T_1) - A_{\xi_1}(dt(T_2)T_2) \\
    & = &  -dn(T_2) m_2 T_1  -dt(T_2)m_1 T_2, \\
  \end{array}
 \]
\[\begin{array}{lcl}
    (\nabla_{T_1} A_{\xi_2})(T_2) & = & \nabla_{T_1}\left(A_{\xi_2}(T_2)\right) - A_{\nabla^{\perp}_{T_1}\xi_2}(T_2) - A_{\xi_2}(\nabla_{T_1}T_2) \\
     & = & - A_{\nabla^{\perp}_{T_1}\xi_2}(T_2) - A_{\xi_2}(\nabla_{T_1}T_2) \\
     & = &  A_{dn(T_1)\xi_1}(T_2) + A_{\xi_2}(dt(T_1)T_1) \\
     & = & dn(T_1) A_{\xi_1}(T_2) + dt(T_1)A_{\xi_2}(T_1) \\
     & = & dn(T_1) m_1T_2 + dt(T_1) m_2 T_1 \\
  \end{array}
 \]
 and
\[\begin{array}{lcl}
    (\nabla_{T_2} A_{\xi_2})(T_1) & = & \nabla_{T_2} \left(A_{\xi_2}(T_1)\right) - A_{\nabla^{\perp}_{T_2}\xi_2}(T_1) - A_{\xi_2}(\nabla_{T_2}T_1) \\
     & = & \nabla_{T_2} \left(A_{\xi_2}(T_1)\right) \\
     & = & \nabla_{T_2} (m_2T_1) \\
     & = & dm_2(T_2)T_1 + m_2 dt(T_2)T_2,\\
  \end{array}
 \]
the Codazzi equations are satisfied if and only if
\[\left\{
  \begin{array}{lclr}
    m_1dt(T_2) &=& -dm_1(T_1) & \hspace{3cm} (C1)\\
    m_1dt(T_1) &=& m_2dn(T_2) & \hspace{3cm} (C2)\\
    m_2dt(T_1) &=& dm_2(T_2) & \hspace{3cm} (C3)\\
    m_2dt(T_2) &=& m_1dn(T_1). & \hspace{3cm} (C4)\\
  \end{array}
\right.\]

Observe that $(C1)$ (resp. $(C3)$) is equivalent to $m_1T_2$ (resp. $m_2 T_1$) be a gradient.

\section{Complete surfaces with constant principal angles.}\label{section complete surfaces}

\subsection{The generic case}
Here we show that a complete surface $\Sigma \subset \mathbb{R}^4$ with constant principal angles $0 < \theta_1 < \theta_2 < \frac{\pi}{2}$ is totally geodesic.

From equation (\ref{dependencia1}) the differential $dt$ is a linear combination of $d\lambda_1,d\lambda_2$. Namely,
\[ dt = Ad\lambda_1 + Bd\lambda_2 \]
where $A,B$ are constants which depend on $\theta_1,\theta_2$. The following lemma is crucial.

\begin{lema} 
The following formulae hold:
\label{eq.dif.} \[ \nabla_{T_2} \nabla \lambda_2 = \|\nabla \lambda_2\| (A \nabla \lambda_1 + B \nabla \lambda_2 ) \]
\[ \nabla_{T_1} \nabla \lambda_1 = -\|\nabla \lambda_1\| (A \nabla \lambda_1 + B \nabla \lambda_2 ). \]
\end{lema}

\begin{proof} Keep in mind that \[ \nabla \lambda_1 = m_1 T_2\hspace{1cm}\mbox{ and }\hspace{1cm}\nabla \lambda_2 = m_2 T_1 \, .\]

For the first equality we have to show that
 \[ \langle \nabla_{T_2} \nabla \lambda_2, T_1 \rangle = \|\nabla \lambda_2\| B \langle \nabla \lambda_2, T_1 \rangle \hspace{1cm} (*) \]
 and
\[ \langle \nabla_{T_2} \nabla \lambda_2, T_2 \rangle = \|\nabla \lambda_2\| A \langle \nabla \lambda_1, T_2 \rangle. \hspace{1cm} (**)\]
We compute 
\[\begin{array}{lcl}
\langle \nabla_{T_2} \nabla \lambda_2, T_2 \rangle & = & \langle \nabla_{T_2} m_2 T_1, T_2 \rangle  \\
 & = & m_2 \langle \nabla_{T_2}T_1, T_2 \rangle  \\
 & = & m_2 \langle dt(T_2)T_2  , T_2 \rangle \\
   & = & m_2\langle (A d\lambda_1(T_2) + B d\lambda_2(T_2)) T_2, T_2 \rangle\\
   & = & m_2 \langle (A d\lambda_1(T_2)) T_2, T_2 \rangle  \\
   & = & m_2 A d\lambda_1(T_2) \\
   & = & \|\nabla \lambda_2\| A \langle \nabla \lambda_1, T_2 \rangle,
\end{array}\]
which proves $(**)$. The proof of $(*)$ is analogous:
\[\begin{array}{lcl}
\langle \nabla_{T_2} \nabla \lambda_2, T_1 \rangle & = & \langle \nabla_{T_1} \nabla \lambda_2, T_2 \rangle  \\
 & = & m_2 \langle \nabla_{T_1}T_1, T_2 \rangle  \\
 & = & m_2 \langle dt(T_1)T_2  , T_2 \rangle \\
   & = & m_2\langle (B d\lambda_2(T_1)) T_2, T_2 \rangle\\
   & = & m_2 \langle (B d\lambda_2(T_1)) T_2, T_2 \rangle  \\
   & = & m_2 B d\lambda_2(T_1) \\
   & = & \|\nabla \lambda_2\| B \langle \nabla \lambda_2, T_1\rangle.
\end{array}\]
The proof of the second equality is analogous and is therefore omitted. This proves the lemma.
\end{proof}

\begin{theorem}
\label{thm:generic-complete}
Assume that $\Sigma \subset \mathbb{R}^4$ with constant principal angles $\theta_1,\theta_2$ such that $0 < \theta_1 < \theta_2 < \frac{\pi}{2}$ is a complete surface. Then $\Sigma$ is totally geodesic.
\end{theorem}
\begin{proof}
Indeed, taking inner product with $\nabla \lambda_2$ in both sides of the first equation of the above lemma we get
\[ \langle \nabla_{T_2} \nabla \lambda_2, \nabla \lambda_2 \rangle = \|\nabla \lambda_2\|B \langle \nabla \lambda_2, \nabla \lambda_2 \rangle. \]
Then along the flow $F_t^{T_2}$ of the vector field $T_2$ the function $f(t) = \|\nabla \lambda_2\|$ satisfies
\[ \left(\frac{f^2}{2}\right)' = B f^3. \]
Observe that $B \neq 0$ because $\theta_2 \neq 0$. Since the flow of $T_2$ is complete it is not difficult to see that $f(t) \equiv 0$.
This shows $m_2 \equiv 0$.

Analogously, taking the inner product with $\nabla \lambda_1$ in both sides of the second equation of the above lemma we get
\[ \langle \nabla_{T_1} \nabla \lambda_1, \nabla \lambda_1 \rangle = -\|\nabla \lambda_1\|A \langle \nabla \lambda_1, \nabla \lambda_1 \rangle. \]
Then along the flow $F_t^{T_1}$ of the vector field $T_1$ the function $g(t) = \|\nabla \lambda_1\|$ satisfies
\[ \left(\frac{g^2}{2}\right)' = -A g^3.\]
Note that $A \neq 0$ because $\theta_1 \neq 0$. Since the flow of $T_1$ is complete we get that $g(t) \equiv 0$. This shows $m_1 \equiv 0$. Thus $\alpha \equiv 0$ and $\Sigma$ is totally geodesic.
\end{proof}

\begin{cor}
\label{cor:symplectomorphism-affine}
Let $f:\mathbb{R}^2 \rightarrow \mathbb{R}^2$ by a symplectomorphism and let $Jf$ be its Jacobian matrix.
If $\| Jf \|$ is a constant function then $f$ is affine.
\end{cor}
\begin{proof}
As explained in Proposition \ref{symplecto} below the graph of such a symplectomorphism can be used to construct a surface with constant principal angles. Since it is an entire graph it is complete, and the above theorem implies that the graph is totally geodesic. Thus, $f$ must be an affine map.
\end{proof}

\subsection{Complete surfaces with $\theta_1 = 0$}

Here we prove that a complete not totally geodesic surface with constant principal angles is a product.

\begin{theorem}
\label{thm:complete-angulo-zero}
Assume $\Sigma \subset \mathbb{R}^4$ with constant principal angles such that $\theta_1=0$
to be complete and not totally geodesic. Then $T_1,T_2$ are parallel vector fields and $\Sigma$ is an extrinsic product.
\end{theorem}
\begin{proof}
If $\theta_2 = \frac{\pi}{2},$ from the first equation in (\ref{tangente1}) we get $dt=0,$ and thus that $T_1$ and $T_2$ are parallel vector fields. So assume $0 < \theta_2 < \frac{\pi}{2}$. As in the proof of Theorem \ref{thm:generic-complete} (using Lemma \ref{eq.dif.} and $B\neq 0$) we get that $\lambda_2$ is a constant function. The second equation in (\ref{zero}) then implies $dt=0,$ and thus that $T_1$ and $T_2$ are parallel vector fields. That $\Sigma$ is an extrinsic product follows from the well-known Moore's Lemma \cite[p. 28]{BCO} since $\alpha(T_1,T_2) \equiv 0$.
\end{proof}

\subsection{Proof of Theorem \ref{main}.}
The first claim was already proved in Lemma \ref{flat}. 
The second claim is a consequence of Theorem \ref{thm:generic-complete} and Theorem \ref{thm:complete-angulo-zero}.
The third part is a consequence of the previous theorem since a compact surface is complete.

\section{Do Carmo-Dajczer-Tojeiro compositions}
\label{sec:compositions}

In \cite{DT1},\cite{DT2} and \cite{DoDa} local isometric immersions of $\mathbb{R}^2$ into $\mathbb{R}^4$ were studied. The authors introduced the concept of \emph{compositions}: an isometric immersion $i: U \subset \mathbb{R}^2 \rightarrow \mathbb{R}^4$ is called a composition (or is said to be trivial) if there exist isometric immersions $g: U \subset \mathbb{R}^2 \rightarrow \mathbb{R}^3$ and $F: W \subset \mathbb{R}^3 \rightarrow \mathbb{R}^4$ such that
\[ i = F \circ g. \]

Since the surfaces $\Sigma \subset \mathbb{R}^4$ with constant principal angles are flat it is natural to understand when they are compositions.

\begin{prop}\label{CNT} Let $\Sigma \subset \mathbb{R}^4$ be a surface with constant principal angles $ 0 < \theta_1 < \theta_2  < \frac{\pi}{2}$.
 The following facts are equivalent:
 \begin{itemize}
\item[(i)] $\Sigma$ is a composition,
\item[(ii)] the first normal space $N_1$ has rank one,
\item[(iii)] either $T_1$ or $T_2$ is a totally geodesic vector field.
 \end{itemize}
\end{prop}
\begin{proof}
We first prove $(i)\Rightarrow (ii)$ by contradiction. Assume that $\Sigma \subset \mathbb{R}^4$ is a composition such that $rank(N_1)=2.$ It follows from \cite[p. 209, Proposition 2.2]{DoDa} that either $dn(T_1) \equiv 0$ or $dn(T_2) \equiv 0$. Then Codazzi equation $(C2)$ or $(C4)$ implies that either $T_1$ or $T_2$ is totally geodesic. Assume first that $T_1$ is totally geodesic. Then equation (\ref{dependencia1}) implies $d \lambda_2 (T_1) \equiv 0,$ which in turn implies $m_2 \equiv 0$ and contradicts the hypothesis $rank(N_1) = 2$. A similar argument shows that if $T_2$ is geodesic then $m_1 \equiv 0$ and so $rank(N_1)=1,$ a contradiction. Thus $(i)$ implies $(ii).$ $(ii)\Rightarrow (i)$ is proved in \cite{DT1}. $(iii)\Rightarrow(ii)$ follows from equation (\ref{dependencia1}) and $(ii)\Rightarrow(iii)$ from Codazzi equations.
\end{proof}

\begin{rem}
\em
A surface $\Sigma \subset \mathbb{R}^4$ with constant principal angles $\{ \theta_1,\theta_2 \} = \{0, \frac{\pi}{2}\} $ is a composition since it is a product. In general, a product has first normal space $N_1$ of rank $2$.
\end{rem}

\begin{prop}Let $\Sigma \subset \mathbb{R}^4$ be a surface with constant principal angles such that $\theta_1 = 0$.
Then $\Sigma$ is a composition.
\end{prop}
\begin{proof} If $rank(N_1) = 1$ then the claim follows from \cite{DT1}. Assume $rank(N_1) = 2$.
Equations (\ref{zero}) imply $dn(T_1) = 0$ and so by \cite[p. 209, Proposition 2.2]{DoDa} we get that $\Sigma$ is a composition.
\end{proof}

\section{Constant angles surfaces in spheres}\label{section surfaces spheres}

A product of two circles is a surface of constant principal angles contained in a sphere. The aim of this section is to classify the surfaces with constant principal angles contained in some sphere.
Along this section we assume $\Sigma \subset \mathbb{R}^4$ to be a surface of constant principal angles with respect to the plane $\Pi \subset \mathbb{R}^4$. The following lemma is well-known:

\begin{lema} The surface $\Sigma$ is contained in some sphere if and only if there exists a normal parallel vector field $\xi$ whose shape operator is a multiple of the identity, i.e. is of the form $A_{\xi} = \delta \mathrm{Id}$ with $\delta\neq 0.$
\end{lema}

Now assume that $\Sigma$ is contained in some sphere. Then there exists a function $\theta$ such that the normal vector field $\xi$ defined as
\[ \xi := \cos(\theta) \xi_1 + \sin(\theta) \xi_2 \]
is normal, parallel and such that $A_{\xi} = \delta \mathrm{Id},$ $\delta\neq 0.$

\begin{lema} The surface $\Sigma \subset \mathbb{R}^4$ with constant principal angles is contained in some sphere if and only if there exist a function $\theta$ and a number  $\delta\neq 0$ such that
\begin{equation}\label{system sphere}
m_1 \cos(\theta)= \delta,\hspace{.5cm}  m_2 \sin(\theta)= \delta\hspace{.5cm}\mbox{and}\hspace{.5cm}d\theta=-dn.
\end{equation}
\end{lema}
\begin{proof} If $\xi= \cos(\theta) \xi_1 + \sin(\theta) \xi_2,$ we have 
$$A_{\xi}=\cos(\theta)A_{\xi_1} + \sin(\theta) A_{\xi_2}=\left(\begin{array}{cc}\sin(\theta)m_2 & 0 \\0 & \cos(\theta)m_1 \\ \end{array} \right).$$
Thus $A_{\xi}=\delta Id$ if and only if the first two equations in (\ref{system sphere}) hold. Moreover, recalling (\ref{def t n}), we get
\[ \begin{array}{lcl}
  \nabla^{\perp}_X \xi & = & \nabla^{\perp}_X \left(\cos(\theta) \xi_1 + \sin(\theta) \xi_2\right)\\
                       
                       & = & - \sin(\theta)\left(d\theta(X) + dn(X)\right) \xi_1 + \cos(\theta)\left(dn(X) + d\theta(X)\right) \xi_2,
\end{array} \]
and the vector fielf $\xi$ is parallel if and only if the last equation in (\ref{system sphere}) holds.
\end{proof}

We now prove by contradiction that $\theta_1=0$ or $\theta_2=\pi/2.$ Equations (\ref{dependencia1})-(\ref{dependencia3})  together with (\ref{system sphere}) and Codazzi equations (C1)-(C2) give the following system for $d\theta(T_1),d\theta(T_2), dt(T_1)$ and $dt(T_2)$:
\begin{equation*}\label{esferico1}\left\{
  \begin{array}{lclr}
 (\frac{\cos(\theta_1)}{\cos(\theta_2)}-\frac{\cos(\theta_2)}{\cos(\theta_1)})dt(T_1)&=&  \frac{\sin(\theta_2)\delta}{\cos(\theta_1)\sin(\theta)} &(1-T_1)\\
  (\frac{\cos(\theta_1)}{\cos(\theta_2)}-\frac{\cos(\theta_2)}{\cos(\theta_1)}) dt(T_2)&=&   \frac{\sin(\theta_1)\delta}{\cos(\theta_2)\cos(\theta)}  &(1-T_2)\\
 \frac{\sin(\theta_1)}{\sin(\theta_2)} d\theta(T_1) + \frac{\cos(\theta_1)}{\cos(\theta_2)}dt(T_1)&=&   \frac{\cos(\theta_1)\delta}{\sin(\theta_2)\sin(\theta)} &(2-T_1)\\
  \frac{\sin(\theta_1)}{\sin(\theta_2)} d\theta(T_2) + \frac{\cos(\theta_1)}{\cos(\theta_2)}dt(T_2)&=&   \frac{\sin(\theta_1)\delta}{\cos(\theta_2)\cos(\theta)} &(2-T_2)\\
   \frac{\sin(\theta_2)}{\sin(\theta_1)} d\theta(T_1) + \frac{\cos(\theta_1)}{\cos(\theta_2)} dt(T_1) &=& 0 &(3-T_1)\\
   \frac{\sin(\theta_2)}{\sin(\theta_1)} d\theta(T_2) + \frac{\cos(\theta_1)}{\cos(\theta_2)} dt(T_2) &=& (\frac{\sin(\theta_1)}{\cos(\theta_2)}- \frac{\cos(\theta_2)}{\sin(\theta_1)})\frac{\delta}{\cos(\theta)} &(3-T_2)\\
   dt(T_2) &=& -\tan(\theta)d\theta(T_1) &(C_1)\\
   dt(T_1) &=& -\cot(\theta)d\theta(T_2) &(C_2)\\
  \end{array}
\right. \, . \end{equation*}

Observe that the values of $d\theta(T_1),d\theta(T_2), dt(T_1)$ and $dt(T_2)$ are given by the first two and the last two equations. Then using equation $(3-T_2)$ we easily get
\[ -2\delta \cos^2(\theta_1)\csc(\theta_1 - \theta_2)\csc(\theta_1 + \theta_2)\sec(\theta_2)\sec(\theta)\sin(\theta_1) = 0. \]
So we have a contradiction unless either $\theta_1 = 0 $ or $\theta_2 = \frac{\pi}{2}$.

Assume that $\theta_1 = 0$ (the case $\theta_2= \frac{\pi}{2}$ is reduced to that case, switching the role of $\Pi$ and $\Pi^{\perp}$). Since $\Sigma$ is flat we can take local coordinates $(x,y)$ where the metric is given by $dx^2 + dy^2$. Using the structure equations (\ref{zero}), we see that the flow lines of $T_2$ are level curves of the functions $\lambda_2$ and $t,$ whereas the flow lines of $T_1$ are the level curves of $\lambda_1,n$ and $\theta.$ Thus the vector field $T_2$ is geodesic and its flow in $(x,y)$ consists of straight lines, and we get a nice intrinsic description of the surface using Fermi coordinates. Namely, let $\gamma(s)$ be the arc length parametrization of a level curve of the function $\theta$ and let $(s,r)$ be the Fermi coordinates around a point of $\gamma$. The flat metric $dx^2 + dy^2$ is expressed in coordinates $(s,r)$ as
\[dr^2 + (1 - r\kappa(s))^2 ds^2 , \]
where $\kappa(s)$ is the curvature of $\gamma(s)$. As explained above the function $\theta$ only depends on $r,$ and the function $t$ only on $s$.  We have
 $$\nabla t  = \frac{t'(s)}{(1 - r\kappa(s))^2 } \frac{\partial}{\partial s}\hspace{1cm}\mbox{and}\hspace{1cm}
 \nabla \theta = \theta'(r) \frac{\partial}{\partial r}.$$
 Moreover, using the second equation in (\ref{zero}), the very definition of $m_2$  and the second equation in (\ref{system sphere}), we get
\[ \| \nabla t \| = \frac{\| \nabla \lambda_2 \|}{\tan(\theta_2)} = \frac{|m_2 |}{\tan(\theta_2)} = \frac{\delta}{\tan(\theta_2)|\sin(\theta(r))|}.\]
Thus 
\[ \frac{|t'(s)|}{|1 - r\kappa(s)|} = \frac{\delta}{\tan(\theta_2)|\sin(\theta(r))|} \, . \]
Taking $r=0$ we see that $t'(s)$ is a constant, $t'_0,$ which is such that
\[ \frac{t'_0}{|1 - r\kappa(s)|} = \frac{\delta}{\tan(\theta_2)|\sin(\theta(r))|};\]
thus $\kappa(s)$ is a constant, $\kappa_0$. This shows that the flow lines of $T_1$ are arcs of circles in coordinates $x,y$. Actually, the flow lines of $T_1$ are arcs of circles in the space $\mathbb{R}^4$. To see this, observe that the flow lines of $T_1$ are contained in a
plane parallel to $\Pi$ because $T_1$ is always tangent to $\Pi$. The absolute value of the curvature of a given flow line of $T_1$ is given by
\[ \| D_{T_1}T_1\|^2 = \| \nabla_{T_1}T_1 \|^2 + \| \alpha(T_1,T_1)\|^2= \kappa_0^2+ \frac{\delta^2}{\sin^2(\theta(r))}, \]
which does not depend on the parameter $s$. So the flow lines of $T_1$ are arcs of circles contained in parallel planes. We have the following result.

\begin{theorem}\label{spherical} Let $\Sigma \subset \mathbb{R}^4$ be a surface with constant principal  angles with respect to the plane $\Pi \subset \mathbb{R}^4$. If $\Sigma$ is contained in a sphere $S^3 \subset \mathbb{R}^4$ then either $\theta_1 = 0$ or $\theta_2 = \frac{\pi}{2}$ and $\Sigma$ is a surface of revolution around a fixed plane $\Pi_0 \subset \mathbb{R}^3 \subset \mathbb{R}^4$ obtained by revolving a spherical helix curve of $\mathbb{R}^3$.
\end{theorem}
\begin{proof}
Let $\mathcal{C}$ be a flow line of $T_1$. As explained above $\mathcal{C}$ is an arc of circle contained in a plane parallel to $\Pi$. So we can assume $\mathcal{C} \subset \Pi$. Let $p \in \mathcal{C}$ and let $\gamma(r) \subset \Sigma \subset \mathbb{R}^4$ be the flow line of $T_2$ starting from $p.$ Notice that the vectors $\xi(r) := \gamma(r) - p$ are always perpendicular to $\mathcal{C}$ at $p$. Indeed, \[ D_{T_2} T_1 = \nabla_{T_2} T_1 + \alpha(T_2,T_1) = 0 + 0 \]
since $dt(T_2)=0$ ($T_2$ is geodesic). This shows that $T_1$ is constant in $\mathbb{R}^4$ along $\gamma(r)$ and thus that the vectors $\xi(r) := \gamma(r) - p$ are always perpendicular to $\mathcal{C}$ at $p$. In other words, the curve $\gamma(r)$ is contained in the normal space $\nu_p(\mathcal{C})$ of the curve $\mathcal{C} \subset \mathbb{R}^4$. The above discussion is independent of the point $p \in \mathcal{C}$. Denote by $\gamma_p(r)$ the corresponding curve in $\nu_p(\mathcal{C})$. Now fix $r = r_0,$ and let $\xi_{r_0}(p) := p - \gamma_p(r_0)$ be a field of normal vectors along $\mathcal{C}$ (assuming that $p$ varies in $\mathcal{C}$). Observe that the derivative of $\gamma_p(r_0)$ in the direction of $T_1$ is a multiple of $T_1$ because $T_1$ is parallel along $\gamma_p$. This shows that the normal vector field $\xi_{r_0}(p)$ is parallel with respect to the normal connection of $\mathcal{C}$. Since $\mathcal{C}$ is a circle contained in $\Pi$ the parallel transport with respect to the normal connection is given by the rotations around $\Pi$ which fix $\Pi^{\perp}$.

Now fix $p_0 \in \mathcal{C}$. Identifying the normal space $\nu_{p_0}(\mathcal{C})$ to $\mathbb{R}^3$ we see that the flow line $\gamma(r)$ of $T_2$ starting at $p_0$ is a classical helix curve of $\mathbb{R}^3$ with respect to $N_{p_0} \in \nu_{p_0}(\mathcal{C}) \cong \mathbb{R}^3$, where
$N_{p_0}$ is the unit normal at $p_0$ of $\mathcal{C} \subset \Pi$. Indeed, $\nu_{p_0}(\mathcal{C}) = \mathbb{R} N_{p_0} \oplus \Pi^{\perp}$ and, as we explained above, $\gamma(r) \subset \nu_{p_0}(\mathcal{C})$. Since $\gamma(r)$ is a helix curve and since it is contained in $\Sigma$ it follows that $\gamma(r)$ is a \emph{spherical helix}; see \cite[p. 248, Lemma 8.22]{GAS} for a complete discussion about such curves.
\end{proof}

\begin{rem}\label{holonomy}
\em
It is interesting to notice that the above proof shows that $\Sigma$ is the union of the so called \emph{holonomy tubes} (see \cite[p. 124]{BCO}) over $\mathcal{C}$ starting with the normal vectors of the curve $\gamma(r)$. A similar construction was used in \cite[Section 5]{DO} to give local examples of non isoparametric immersions of spheres with curvature normals of constant length.
\end{rem}

\section{Structure of constant angle surfaces when $\theta_1  = 0$}\label{section structure}

Here we show that any surface $\Sigma$ with constant principal angles $\theta_1=0$ and arbitrary $\theta_2$ with respect to a plane $\Pi$ is (up to rigid motion) a union of holonomy tubes over a plane curve $\mathcal{C} \subset \Pi$. More precisely, there exists a curve $\gamma \subset \nu_p(\mathcal{C})$ such that $\Sigma$ is locally the union of the holonomy tubes of $\mathcal{C}$ through the points of $\gamma$.

\begin{theorem} Any surface $\Sigma$ with constant principal angles $\theta_1=0$ and arbitrary $\theta_2$ with respect to a plane $\Pi$ is (up to rigid motion) a union of holonomy tubes over a plane curve $\mathcal{C} \subset \Pi$. Moreover, if $\gamma \subset \nu_p(\mathcal{C})$ is the curve of the starting points of the tubes, then $\gamma$ is a helix curve of $\mathbb{R}^3 \cong \nu_p(\mathcal{C}).$ 
\end{theorem}

\begin{proof}
Observe that the structure equations (\ref{zero}) imply that $T_2$ is a geodesic vector field. Now, as in the proof of Theorem \ref{spherical} we introduce flat coordinates $x,y$ and Fermi coordinates $r,s$, where $s$ is the arc length parameter of a level curve $\mathcal{C}$ of the function $n$, i.e. a flow line of the vector field $T_1$. Since $T_1$ is always tangent to $\Pi$ we can assume that $\mathcal{C} \subset \Pi$. Fix $p_0 \in \mathcal{C} \subset \Pi$ and let $\gamma(r) \subset \mathbb{R}^4$ be the flow line of $T_2$ starting from $p_0.$
So $\gamma(r) \subset \nu_{p_0}(\mathcal{C})$ since $T_1$ is constant in $\mathbb{R}^4$ along $\gamma(r)$. Then the same argument as in the proof of Theorem \ref{spherical} shows that the holonomy tubes through $\gamma(r)$ are contained in $\Sigma$. Thus, $\Sigma$ is locally the union of such holonomy tubes. Finally, notice that $\gamma \subset \mathbb{R}^3 \cong \nu_{p_0}(\mathcal{C})$ is a helix curve with respect to $\Pi^{\perp} \subset \nu_{p_0}(\mathcal{C}),$ i.e. is a curve of constant slope with respect to the normal vector $N_{p_0}$ of $\mathcal{C}$ at $p_0$.
\end{proof}

\begin{rem}
\em
Actually any surface $\Sigma$ with constant principal angles with either $T_1$ or $T_2$ geodesic can be described as a union of holonomy tubes over a curve $\mathcal{C}$ by using a helix curve $\gamma_p \subset \nu_p(\mathcal{C})$ as starting points of the holonomy tubes; moreover, the point $p$ may be arbitrarily chosen. Notice that the helices $(\gamma_p)_{p\in C}$ are all congruent since the parallel transport with respect to the normal connection is an isometry between the normal spaces. So, surfaces with a geodesic vector field $T_1$ or $T_2$ are union of congruent helix curves.
\end{rem}

\section{Existence}\label{Ex}

In this section we discuss the existence of non-trivial surfaces $\Sigma \subset \mathbb{R}^4$ with
prescribed constant principal angles with respect to a plane $\Pi$. More precisely, we seek a surface $\Sigma$ such that $\theta_1$ and $\theta_2$ belong to $(0,\pi/2);$ $\Sigma$ is thus locally the graph of a function $F: \Pi \rightarrow \Pi^{\perp}$. Taking orthonormal bases of $\Pi$ and $\Pi^{\perp},$ and writing $F(x,y)=(f(x,y),g(x,y))$ in these bases, the surface $\Sigma\subset \mathbb{R}^4$ is locally parametrized by
\[ (x,y) \rightarrow (x,y,f(x,y),g(x,y)) .\]

Let $\partial_x := (1,0,f_x(x,y),g_x(x,y)), \partial_y := (0,1,f_y(x,y),g_y(x,y))$ be a basis of the tangent plane $\mathrm{T}_{(x,y)}\Sigma$.
As usual, let $\mathrm{E},\mathrm{F},\mathrm{G}$ denote the coefficients of the metric tensor of $\Sigma$ in the coordinates $(x,y),$ i.e.
\[\begin{array}{ccccl}
 \mathrm{E} &:=& \langle \partial_x,\partial_x \rangle &=& 1 + f_x^2 + g_x^2 \, \\
 \mathrm{F} &:=& \langle \partial_x,\partial_y \rangle &=& f_xf_y + g_xg_y \, \\
 \mathrm{G} &:=& \langle \partial_y,\partial_y \rangle &=& 1+ f_y^2 + g_y^2 .
 \end{array}\]
Set $W := \mathrm{T}_{(x,y)}\Sigma$. Then the symmetric operator $\mathcal{S}_{\Pi W}$ defined Section \ref{section principal angles} satisfies
$$\langle S_{\Pi W}\left(\partial_x\right),\partial_x\rangle=\langle S_{\Pi W}\left(\partial_y\right),\partial_y\rangle=1\hspace{.5cm}\mbox{and}\hspace{.5cm}\langle S_{\Pi W}\left(\partial_x\right),\partial_y\rangle=0.$$
So the matrix of $\mathcal{S}_{\Pi W}$ in $(\partial_x,\partial_y)$ is the inverse of the matrix of the metric, i.e. 
\[\mathcal{S}_{\Pi W} = \left(
     \begin{array}{cc}
       \mathrm{E} & \mathrm{F} \\
       \mathrm{F} & \mathrm{G} \\
     \end{array}
   \right)^{-1}.
\]
We deduce the following proposition.

\begin{prop} 
With the notation above, $\Sigma$ has constant principal angles with respect to the plane $\Pi$ if and only if the matrix tensor $\displaystyle{\left(
     \begin{array}{cc}
       \mathrm{E} & \mathrm{F} \\
       \mathrm{F} & \mathrm{G} \\
     \end{array}
   \right)}$ has constant eigenvalues.
\end{prop}

\begin{cor}\label{iguales} Let $\Sigma\subset \mathbb{R}^4$ be a surface with constant principal angles with respect to a plane $\Pi$. If the principal angle $\theta_1$ has multiplicity $2$ then  $\Sigma$ is totally geodesic, i.e. is an open subset of a $2$-plane.
\end{cor}
\begin{proof} If the principal angle $\theta_1$ has multiplicity $2$ then the immersion is orthogonal, i.e. $\mathrm{F} \equiv 0$. Thus $\mathrm{E}$ and $\mathrm{G}$ are constants. This implies that there exist constants $c_1,c_2$ such that the map $(x,y) \rightarrow (c_1f(x,y),c_2g(x,y))$ is a local isometry of $\mathbb{R}^2$. Thus $f,g$ are linear functions and $\Sigma$ is totally geodesic.
\end{proof}

Since $\mathcal{S}_{\Pi W}$ has constant eigenvalues if and only if its characteristic polynomial does not depend on the point $(x,y),$ we get:
\begin{prop}\label{graficoHelice} 
With the notation above, $\Sigma$ has constant principal angles $\theta_1,\theta_2\in (0,\pi/2)$ with respect to the plane $\Pi$ if and only if
\[\begin{array}{lcl}
  \mathrm{E} + \mathrm{G} &=& \sec^2(\theta_1) + \sec^2(\theta_2)\\
  \mathrm{EG} - \mathrm{F}^2 &=& \sec^2(\theta_1)\sec^2(\theta_2).
\end{array}\]
\end{prop}

Since \[ \mathrm{EG} - \mathrm{F}^2 = 1 + f_x^2 + g_x^2 + f_y^2 + g_y^2 + (f_x g_y - f_yg_x)^2, \]
in order to show the existence of surfaces with constant principal angles we have to solve the following PDE system:
\[ \left\{
  \begin{array}{lcl}
   2 + f_x^2 + g_x^2 + f_y^2 + g_y^2 &=& \sec^2(\theta_1) + \sec^2(\theta_2)\\
  (f_x g_y - f_yg_x)^2 &=& \sec^2(\theta_1)\sec^2(\theta_2) + 1 -\sec^2(\theta_1) - \sec^2(\theta_2) \\
   & = & \tan^2(\theta_1)\tan^2(\theta_2).
  \end{array}
\right. \]

Taking the square root we get the equivalent system
\[ \left\{
  \begin{array}{lcl}
  f_x^2 + g_x^2 + f_y^2 + g_y^2 &=& c_1 := \sec^2(\theta_1) + \sec^2(\theta_2) - 2,\\
  f_x g_y - f_yg_x &=& c_2:= \tan(\theta_1)\tan(\theta_2).
  \end{array}
\right. \]

Since $c_2 \neq 0$ we can divide both hands by $c_2$ and obtain the following proposition.

\begin{prop} \label{symplecto} There exists a non totally geodesic surface with constant principal angles if and only if there exists a non linear local symplectomorphism of $\mathbb{R}^2$ whose Jacobian matrix has constant length.
\end{prop}
\begin{proof}
Assume that such a surface exists, and set $\psi(x,y) := (\frac{f}{\sqrt{c_2}}, \frac{g}{\sqrt{c_2}})$.
Then the system above implies that $\psi$ is a non linear local symplectomorphism of $\mathbb{R}^2$ whose Jacobian matrix has constant length. The converse follows from the equivalences above.
\end{proof}

\subsection{Construction of local symplectomorphisms}

We prove the following existence result:

\begin{theorem}
\label{hay}
There exist non linear local symplectomorphisms of $\mathbb{R}^2$ whose Jacobian matrix has constant length. Thus, there exist non totally geodesic surfaces with different constant principal angles.
\end{theorem}

As explained above we have to show the existence of non linear solutions of the system
\begin{equation}\label{system c2=1}
 \left\{
  \begin{array}{lcl}
  f_x^2 + g_x^2 + f_y^2 + g_y^2 &=& c_1 ,\\
  f_x g_y - f_yg_x &=& 1
  \end{array}
\right. 
\end{equation}
where $c_1 \neq 2$ (if $c_1=2,$ by Corollary \ref{iguales} the solutions are linear functions).

Assume that $f$ is  a (non constant) known function and set $A = g_x, B=g_y$. Then from the second equation in (\ref{system c2=1}) there exists a function $\lambda$ such that
$$A =\frac{-f_y + \lambda f_x}{f_x^2 + f_y^2}\hspace{1cm}\mbox{ and }\hspace{1cm}B=\frac{f_x + \lambda f_y}{f_x^2 + f_y^2}  .$$
Using the first equation in (\ref{system c2=1}), $\lambda$ satisfies
\[ \frac{1+ (f_x^2 + f_y^2)^2 +\lambda^2}{f_x^2 + f_y^2} = c_1,\]
and setting $\Delta := f_x^2 + f_y^2,$ this last equation reads
\[ \lambda^2= c_1 \Delta - 1 - \Delta^2 \, .\]

\begin{rem}
\em
Notice that $c_1 \geq 2,$ and that $c_1 = 2$ implies $\lambda \equiv 0$ and $\Delta \equiv 1$.
So if $\lambda \equiv 0$ and $\Delta \equiv 1$ then $f$ is harmonic with a gradient of constant length.
Then it is not difficult to show (using the theory of complex functions) that $f$ must be linear. This gives another proof of Corollary \ref{iguales}.
\end{rem}

\begin{rem}
\em
The norm of the gradient $\Delta$ is necessarily bounded by
\[ \frac{ c_1 - \sqrt{c_1^2 - 4} }{2} \leq \Delta \leq \frac{ c_1 + \sqrt{c_1^2 - 4} }{2}.\]
\end{rem}

Assume that $c_1 > 2$. Then $f$ is a solution of the following second order quasi-linear PDE:

\begin{equation}
\left(\frac{-f_y + \sqrt{c_1 \Delta - 1 - \Delta^2} f_x}{\Delta}\right)_y =\left (\frac{f_x + \sqrt{c_1 \Delta - 1 - \Delta^2} f_y}{\Delta}\right)_x.
\end{equation}

Reciprocally, if $f$ is a non linear solution of the above equation then it is possible to construct $g$ such that $F(x,y)=(f(x,y),g(x,y))$ is a (non linear) symplectomorphism whose Jacobian matrix has constant length. The existence of such symplectomorphism, i.e. the proof of Theorem \ref{hay}, follows from the following theorem.

\begin{theorem}\label{existen} Let 
\[ \mathrm{L}f = A(f_x,f_y) f_{xx} + B(f_x,f_y) f_{xy} + C(f_x,f_y) f_{yy} + E(f_x,f_y)\]
be a second order quasi-linear operator with analytic coefficients $A,B,C,E$ defined in some open subset $\Omega \subset \mathbb{R}^2$. Then there exists a non linear function $f: D \subset \mathbb{R}^2 \rightarrow \mathbb{R}$ such that $\mathrm{L}f=0$.
\end{theorem}
\begin{proof}
We can apply the Cauchy-Kowalewski theorem to solve the equation $\mathrm{L}f=0$ as soon as we can find non characteristic real analytic initial data (see \cite[p. 56]{J} for details). Let us set, as it is standard, $p = f_x, q = f_y$. Let $g$ be the quadratic differential form defined on $\Omega \subset \mathbb{R}^2$ by
 \[ g = Adq^2 + Bdpdq + C dp^2 \, \, . \]

Notice that $g$ is not identically zero since $\mathrm{L}$ is a differential operator of second order. So we can find an analytic vector field $V$ such that
 \[ g(V,V) \neq 0 \, \, \]
 around a point $I_0 = (p_0,q_0) \in \Omega$. We can also consider $(p_0,q_0)$ as a point in the $(x,y)$ plane. So $V(x,y)$ is also a vector field around $(p_0,q_0)$ in the plane $(x,y)$. It is not difficult to see that there exists a non constant analytic curve $\gamma(t)$ such that $\gamma(0) = (p_0,q_0)$ and $\langle \gamma'(t), V(\gamma(t)) \rangle = 0$, where $\langle,\rangle$ is the standard scalar product of $\mathbb{R}^2,$ i.e. such that $V$ is normal to $\gamma$.  Consider the following initial conditions on $\gamma(t)$ for the Cauchy problem for the quasi-linear PDE: 
 \[ f(t) = \frac{\langle \gamma(t), \gamma(t) \rangle}{2} \hspace{1cm}\mbox{and}\hspace{1cm}\frac {\partial f}{\partial V}(t) = \langle \gamma(t), V(\gamma(t)) \rangle.\]
Then these initial conditions are analytic and non characteristic. Indeed, the condition $g(V,V) \neq 0$ holds for the initial data $(f(t),  \frac {\partial f}{\partial V}(t) )$ (since $(\nabla f)|_{\gamma(t)} = \gamma(t)$) and this is exactly the condition on the initial data to be non characteristic. Thus we can apply the Cauchy-Kowalewski theorem to get a solution $f$ around $(p_0,q_0)$.  Moreover $\nabla f$ is not constant since $(\nabla f)|_{\gamma(t)} = \gamma(t)$. \end{proof}

\section{Existence of surfaces without geodesic principal directions}\label{section surfaces without geodesic}

Here we show that there exist surfaces with constant principal angles such that both vector fields $T_1$ or $T_2$ are not geodesic.

\subsection{The rank of the first normal space of a graph.}

Assume that $\Sigma$ is parametrized by $F(x,y) = (x,y,f(x,y),g(x,y)).$ The tangent space at the point $F(x,y)$ is generated by the vectors
$$F_x= (1,0,f_x,g_x)\hspace{1cm}\mbox{and}\hspace{1cm}F_y= (0,1,f_y,g_y),$$
and the normal space by the vectors
$$N_1=(-f_x,-f_y,1,0)\hspace{1cm}\mbox{and}\hspace{1cm}N_2=(-g_x,-g_y,0,1).$$
The first normal space of the graph is generated by the normal vectors 
$$\alpha(\partial_x,\partial_x)=(F_{xx})^{\perp},\hspace{1cm}\alpha(\partial_x,\partial_y)=(F_{xy})^{\perp}\hspace{.5cm}\mbox{ and }\hspace{.5cm}\alpha(\partial_y,\partial_y)=(F_{yy})^{\perp},$$ 
where $(\cdot)^{\perp}$ means the normal component.

\begin{lema}\label{N1} The first normal space of the graph $F(x,y)$ has rank one if and only if the matrix
\[\left(
  \begin{array}{ccc}
    f_{xx} & f_{xy} & f_{yy} \\
    g_{xx} & g_{xy} & g_{yy} \\
  \end{array}
\right)\]
has rank one.
\end{lema}
\begin{proof} 
Writing
$$(F_{xx})^{\perp}=A_{xx} N_1 + B_{xx} N_2,\hspace{1cm} (F_{xy})^{\perp}=A_{xy} N_1 + B_{xy} N_2$$
and 
$$(F_{yy})^{\perp}=A_{yy} N_1 + B_{yy} N_2,$$
by straightforward computations we get
\[ \left(
     \begin{array}{ccc}
       \langle F_{xx}, N_1 \rangle & \langle F_{xy}, N_1 \rangle & \langle F_{yy}, N_1 \rangle \\
       \langle F_{xx}, N_2 \rangle  & \langle F_{xy}, N_2 \rangle  & \langle F_{yy}, N_2 \rangle \\
     \end{array}
   \right) = G . \left(
     \begin{array}{ccc}
       A_{xx} & A_{xy} & A_{yy} \\
       B_{xx} & B_{xy} & B_{yy} \\
     \end{array}
   \right)\]
where 
\[ G = \left(
\begin{array}{cc}
\langle N_1, N_1 \rangle & \langle N_1, N_2 \rangle \\
\langle N_1, N_2 \rangle & \langle N_2, N_2 \rangle \\
\end{array}
\right)
\]
is the Gram matrix of $N_1,N_2$. Since the first normal space of the graph $F(x,y)$ has rank one if and only if the matrix
\[\left(
     \begin{array}{ccc}
       A_{xx} & A_{xy} & A_{yy} \\
       B_{xx} & B_{xy} & B_{yy} \\
     \end{array}
   \right) \]
 has rank one, and since the Gram matrix $G$ is invertible, the first normal space of the graph $F(x,y)$ has rank one if and only if
 the matrix \[ \left(
     \begin{array}{ccc}
       \langle F_{xx}, N_1 \rangle & \langle F_{xy}, N_1 \rangle & \langle F_{yy}, N_1 \rangle \\
       \langle F_{xx}, N_2 \rangle  & \langle F_{xy}, N_2 \rangle  & \langle F_{yy}, N_2 \rangle \\
     \end{array}
   \right) = \left(
  \begin{array}{ccc}
    f_{xx} & f_{xy} & f_{yy} \\
    g_{xx} & g_{xy} & g_{yy} \\
  \end{array}
\right)\] has rank one.
\end{proof}
\begin{lema}\label{N2} Let $f,g : \Omega \subset \mathbb{R}^2 \rightarrow \mathbb{R}$ be two analytic functions, where $\Omega$ is open and connected, such that the matrix
\[\left(
  \begin{array}{ccc}
    f_{xx} & f_{xy} & f_{yy} \\
    g_{xx} & g_{xy} & g_{yy} \\
  \end{array}
\right)\]
has rank one for all point in $\Omega$. Assume that $f$ is not an affine function. Then one of the following conditions holds:
\begin{itemize}
\item The functions $f,g,1$ are linearly dependent over $\mathbb{R}$;
\item The determinant of the Hessian of $f$ vanishes identically.
\end{itemize}
\end{lema}

\begin{proof} Notice that the matrix \[\left(
  \begin{array}{ccc}
    f_{xx} & f_{xy} & f_{yy} \\
    g_{xx} & g_{xy} & g_{yy} \\
  \end{array}
\right)\] has rank one if and only if there exists a function $M:\Omega \rightarrow \mathbb{R}$ such that
\[ \left\{
     \begin{array}{lcl}
       M df_x & = & dg_x \\
       M df_y & = & dg_y. \\
     \end{array}
   \right.
\]

If $M$ is constant in some open subset of $\Omega$ then $f,g,1$ are linearly independent over $\mathbb{R}$.
So assume that $M$ is not constant. Taking the exterior derivative we get
\[ \left\{
     \begin{array}{lcl}
       dM  \wedge df_x & = & 0 \\
       dM  \wedge df_y & = & 0.\\
     \end{array}
   \right.
\]

Since $dM \neq 0$ we get $df_x\wedge df_y\equiv 0,$ which implies that the determinant of the Hessian of $f$ vanishes identically.
\end{proof}

\subsection{The existence revisited}

Recall from Section \ref{Ex} that to construct a graph with constant principal angles we need a solution of the partial differential equation
\begin{equation}\label{pde} \left(\frac{-f_y + \sqrt{c_1 \Delta - 1 - \Delta^2} f_x}{\Delta}\right)_y =\left(\frac{f_x + \sqrt{c \Delta - 1 - \Delta^2} f_y}{\Delta}\right)_x
\end{equation}
where $\Delta := f_x^2 + f_y^2$. We are going to show that there exist analytical solutions $f$ with non vanishing determinant $\left|
                                                                                   \begin{array}{cc}
                                                                                     f_{xx} & f_{xy} \\
                                                                                     f_{xy} & f_{yy} \\
                                                                                   \end{array}
                                                                                 \right|$.
According to Lemmas \ref{N1} and \ref{N2}, this will imply that the first normal spaces of the graph of $(f,g)$ have rank 2 (the function $g$ is arbitrary). Finally, Proposition \ref{CNT} yields that in this example neither of the two vector fields $T_1$ or $T_2$ is geodesic.

So we are going to prove the following theorem and its corollary.

\begin{theorem}
\label{thm:existence-revisited}
There exist surfaces $\Sigma \subset \mathbb{R}^4$ with constant principal angles $0 < \theta_1 < \theta_2 < \frac{\pi}{2}$
such that their first normal spaces have rank $2$. For such a surface the vector fields $T_1$ and $T_2$ are not geodesic.
\end{theorem}
\begin{cor} There exist surfaces $\Sigma \subset \mathbb{R}^4$ with constant principal angles $0 < \theta_1 < \theta_2 < \frac{\pi}{2}$
which are not compositions in the sense of Do Carmo-Dajczer.
\end{cor}
\begin{proof}
As explained above it is enough to show the existence of a solution $f$ of the PDE (\ref{pde}) with non vanishing determinant $\left|
                                                                                   \begin{array}{cc}
                                                                                     f_{xx} & f_{xy} \\
                                                                                     f_{xy} & f_{yy} \\
                                                                                   \end{array}
                                                                                 \right|$.
Let $E(u,v)$ be the analytic function defined by
\[ E(u,v) := \frac{u + \sqrt{c_1 \Delta - 1 - \Delta^2} v}{\Delta} \]
where $\Delta := u^2 + v^2;$ equation (\ref{pde}) thus reads
\[ E(-f_y,f_x)_y = E(f_x,f_y)_x,\]
or
\[ E_u(f_x,f_y)f_{xx} + (E_v(f_x,f_y) - E_v(-f_y,f_x))f_{xy} + E_u(-f_y,f_x)f_{yy} = 0.\]
We look for a solution of this equation with the initial conditions
\begin{equation*}
 f(x,0)= \phi(x)\hspace{1cm}\mbox{and}\hspace{1cm}f_y(x,0)= \psi(x),
\end{equation*}
where $\psi$ and $\phi$ are two functions such that
\begin{equation}\label{condition non char}
    E_u(\phi'(x),\psi(x)) \neq 0,\hspace{.5cm}E_v(-\psi(x), \phi'(x))\neq 0\hspace{.5cm}\mbox{and}\hspace{.5cm}\phi''(x)\neq 0. 
    \end{equation}
This means that the initial conditions are not characteristics; thus an analytic solution $f$ of equation (\ref{pde}) does exist by the Cauchy-Kowalewski theorem. Moreover, since $f_{xy}(x,0) \equiv 0$ and $f_{xx}(x,0).f_{yy}(x,0) \neq 0$ we see that the Hessian of $f$ does not vanish and thus that the rank of the first normal space is equal to $2$. In particular the graph is not totally geodesic.

We finally prove the existence of $\psi$ and $\phi$ such that conditions (\ref{condition non char}) hold. Notice that the function $E(u,v)$ is defined in the annulus $ \mathcal{A} = \{ (u,v) : \frac{ c_1 - \sqrt{c_1^2 - 4} }{2} \leq \Delta \leq \frac{ c_1 + \sqrt{c_1^2 - 4} }{2} \},$ where it is moreover analytic. Also notice that $E_u$ and $E_v$ do no vanish identically. Let $\mathcal{B} \subset \mathcal{A}$ be the subset such that either $E_u(u,v) = 0$ or $E_v(-v,u)= 0$. Then $\mathcal{B} \neq \mathcal{A}$ and since $\mathcal{B}$ is closed there exists an open disc $ D \subset \mathcal{A}$ such that $\mathcal{B} \bigcap D = \emptyset$. If $u_0,v_0$ are the coordinates of the center of $D,$ the functions $\phi(x) = x^2 + xu_0$ and $\psi(x) = x + v_0$ for small values of $x$ satisfy the system of initial conditions $(\ref{condition non char})$.
\end{proof}

\subsection{Deformations}

Here we explain how to use a solution $f$ of the PDE (\ref{pde}) in order to locally construct a surface with constant principal angles $\theta_1, \theta_2$. Actually, we will show that a solution $f$ produces a one parameter \emph{deformation} $\Sigma_m \subset \mathbb{R}^4$ of flat surfaces with constant principal angles.

Let $f$ be a solution of the PDE (\ref{pde}). Then there exists $g$ such that 
$$g_x = \frac{-f_y + \sqrt{c_1 \Delta - 1 - \Delta^2} f_x}{\Delta}\hspace{1cm} \mbox{and}\hspace{1cm}g_y= \frac{f_x + \sqrt{c \Delta - 1 - \Delta^2} f_y}{\Delta};$$
so we have 
\[ \left\{
  \begin{array}{lcl}
  f_x^2 + g_x^2 + f_y^2 + g_y^2 &=& c \, \, ,\\
  f_x g_y - f_yg_x &=& 1.
  \end{array}
\right. \]
Let $F_m(x,y) := (m f(x,y), m g(x,y)),$ with $m \in \mathbb{R}$. As explained in Proposition \ref{graficoHelice}, the graph of $F_m$ is a surface with constant principal angles $\theta_1,\theta_2$ if and only if
\[ \left\{ \begin{array}{rcl}
  2 + m^2(f_x^2 + g_x^2 + f_y^2 + g_y^2) &=& \sec^2(\theta_1) + \sec^2(\theta_2) \, \\
  1 + m^2(f_x^2 + g_x^2 + f_y^2 + g_y^2) + m^4(f_x g_y - f_yg_x)^2 &=& \sec^2(\theta_1)\sec^2(\theta_2),
\end{array} \right. \]
i.e. if and only if
\[ \left\{ \begin{array}{rcl}
  2 + m^2c &=& \sec^2(\theta_1) + \sec^2(\theta_2) \, \\
  1 + m^2c + m^4 &=& \sec^2(\theta_1)\sec^2(\theta_2).
\end{array} \right. \]
Thus the graph of $F_m$ is a surface $\Sigma_m$ with constant principal angles 
\[ \left\{ \begin{array}{lcl}
  \theta_1 &=& \mathrm{arcsec}\sqrt{1+\frac{m^2}{2}(c - \sqrt{c^2 - 4})}  \, \\
  \theta_2 &=& \mathrm{arcsec}\sqrt{1+\frac{m^2}{2}(c + \sqrt{c^2 - 4})}\, .
\end{array} \right. \]

It is not difficult to see that the map $(m,c) \rightarrow (\theta_1, \theta_2)$ from $(0,+\infty) \times (2,+\infty)$
into $\Delta = \{(\theta_1, \theta_2) : 0 < \theta_1 < \theta_2 < \frac{\pi}{2} \}$ is surjective.
This shows how to construct the surface $\Sigma_m$ with constant principal angles $\theta_1,\theta_2$ by starting with a solution $f$ of the PDE (\ref{pde}).

\section{Surfaces with parallel mean curvature}

Here we classify surfaces $\Sigma \subset \mathbb{R}^4$ having parallel mean curvature vector field $\mathrm{H}$. Recall that the mean curvature vector $\mathrm{H}$ is the normal vector field given by
\[ \mathrm{H} := \frac{\alpha(T_1,T_1) + \alpha(T_2,T_2)}{2}.\]
We have the following result.
\begin{theorem}
\label{thm:parallel-meancurvature}
Let $\Sigma \subset \mathbb{R}^4$ be a surface with constant principal angles.
The mean curvature vector $\mathrm{H}$ is parallel if and only if $\Sigma$ is a product. Hence, up to a rigid motion, $\Sigma$ is either an open subset of a torus $S_1(r_1) \times S_2(r_2) \subset \mathbb{R}^2 \times \mathbb{R}^2 \subset \mathbb{R}^4$ or an open subset of the cylinder $\mathbb{R} \times \gamma \subset \mathbb{R} \times \mathbb{R}^3 \subset \mathbb{R}^4$, where $\gamma \subset \mathbb{R}^3$ is an helix, i.e. a curve with constant curvature and torsion.
\end{theorem}

\begin{proof}
It is enough to show that $T_1$ and $T_2$ are geodesic vector fields. Indeed, since $\alpha(T_1,T_2) \equiv 0,$ Moore's Lemma implies that $\Sigma$ is a local product. The vector field $\mathrm{H}$ is parallel if and only if
\begin{equation}\label{Hpa}
dm_2 = -m_1 dn \hspace{1cm}\mbox{and}\hspace{1cm} dm_1=m_2 dn. 
\end{equation}
If $m_1$ and $m_2$ are constant functions then the Codazzi equations imply that both $T_1,T_2$ are geodesic vector fields. So we can assume that $m_1$ or $m_2$ is not constant. Observe that the above equations imply that if one is not constant so is the other.
So we have that both functions $m_1$ and $m_2$ are not constant. Now we have
\[ \begin{array}{lcll}
  dt(T_1) &=& \frac{dm_2(T_2) }{m_2}    & \mbox{ by Codazzi equation (C3)} \\
          &=& \frac{-m_1 dn(T_2) }{m_2} & \mbox{ by the first equation in (\ref{Hpa})} \\
          &=& \frac{-m_1 m_1 dt(T_1) }{m_2^2} & \mbox{ by Codazzi equation (C2)} \, .\\
\end{array} \]
Thus $dt(T_1) \equiv 0$ i.e. $T_1$ is a geodesic vector field. In a similar way we have
\[ \begin{array}{lcll}
  dt(T_2) &=& \frac{-dm_1(T_1) }{m_1}    & \mbox{ by Codazzi equation (C1)} \\
          &=& \frac{-m_2 dn(T_1) }{m_1} & \mbox{ by the second equation in (\ref{Hpa})} \\
          &=& \frac{-m_2 m_2 dt(T_2) }{m_1^2} & \mbox{ by Codazzi equation (C4)} \, .\\
\end{array} \]
Thus $dt(T_2) \equiv 0$ i.e. $T_2$ is a geodesic vector field.
\end{proof}

%\section*{Acknowledgements}

%\medskip
\noindent {\bf Authors' Addresses:}

\noindent P. Bayard, \\ Instituto de F\'\i sica y Matem\'aticas, Universidad Michoacana, \\
Edif. C-3, Cd. Universitaria, C.P. 58040 Morelia, Mich. M\'exico \\
  email:     bayard@ifm.umich.mx \\

\noindent A. J. Di Scala, \\ Dipartimento di Matematica, Politecnico di Torino, \\
Corso Duca degli Abruzzi 24, 10129 Torino, Italy \\
  email:     antonio.discala@polito.it \\

\noindent O. Osuna Castro, \\ Instituto de F\'\i sica y Matem\'aticas, Universidad Michoacana, \\
Edif. C-3, Cd. Universitaria, C.P. 58040 Morelia, Mich. M\'exico \\
  email:     osvaldo@ifm.umich.mx \\

\noindent G. Ruiz-Hern\'andez, \\ Instituto de Matem\'aticas, Universidad Nacional Autonoma de M\'exico, \\
Ciudad Universitaria, C.P. 04510 D.F. M\'exico \\
  email:     gruiz@matem.unam.mx \\


\begin{thebibliography}{10}

\bibitem{BCO}
J. Berndt, S. Console and C. Olmos, {\it Submanifolds and
holonomy, } Chapman \& Hall/CRC , Research Notes in Mathematics
434 (2003).


\bibitem{Chen}
B. Y. Chen,
\newblock {\em Geometry of submanifolds},
\newblock  Pure and Applied Mathematics, {\bf 22} (1973). MR0353212

\bibitem{ChenCMC} B. Y. Chen,
\newblock {\it On the surface with parallel mean curvature vector},
\newblock Indiana Univ. Math. J. {\bf 2} (1973), 655-666.

\bibitem{DT1}
M. Dajczer and R. Tojeiro,
{\it Submanifolds with nonparallel first normal bundle,}
Canad. Math. Bull. {\bf 37:3} (1994), 330-337.

\bibitem{DT2}
M. Dajczer and R. Tojeiro,
{\it On flat surfaces in space forms,}
Houston J. Math.  {\bf 21:2}  (1995), 319-338.

\bibitem{DoDa}
M. Dajczer and M. do Carmo,
{\it Local isometric immersions of $\mathbb{R}^2$ into $\mathbb{R}^4$,}
J. Reine Angew. Math.  {\bf 442}  (1993), 205-219.


\bibitem{DO}
A. J. Di Scala and C. Olmos, {\it Submanifolds with curvature normals of constant length and the Gauss map, }
J. Reine Angew. Math.  {\bf 574}  (2004), 79-102.


\bibitem{DS}
A. Di Scala,
\newblock {\em Weak helix submanifolds of euclidean spaces},
\newblock Abh. Math. Semin. Univ. Hambg. {\bf 79} (2009), 37-46.

\bibitem{DS-RH}
A. Di Scala and G. Ruiz-Hern\'andez,
\newblock {\em Helix submanifolds of euclidean space},
\newblock  Monatsh. Math. {\bf 157} (2009), 205-215.

\bibitem{DS-RH-II}
A. Di Scala and G. Ruiz-Hern\'andez,
\newblock {\em Higher codimensional Euclidean helix submanifolds},
\newblock  Kodai Math. J.  {\bf 33:2}  (2010), 192-210.

\bibitem{Di-Mu}
F. Dillen and M. I. Munteanu,
\newblock {\em Constant Angle Surfaces in
$\mathbb{H}^2 \times \mathbb{R}$, }
\newblock  Bull. Braz. Math. Soc. {\bf 40:1}  (2009), 85-97.

\bibitem{DFV}
F. Dillen, J. Fastenakels, J. Van der Veken and L. Vrancken,
\newblock {\it Constant Angle Surfaces in
$\mathbb{S}^2\times \mathbb{R}$, }
\newblock Monatsh. Math. {\bf 152:2} (2007), 89-96. MR2346426

\bibitem{GAS}
A. Gray, E. Abbena, S. Salamon, {\it Modern Differential Geometry of
CURVES and SURFACES with Mathematica , } CRC Press, (2006) .

\bibitem{Jiang} S. Jiang, R.
\newblock {\it Angles between Euclidean subspaces},
\newblock Geometriae Dedicata {\bf 63:2} (1996), 113-121.

\bibitem{J}
F. John, {\it Partial Differential Equations, } Fourth Edition, Applied Mathematical
Sciences, Springer-Verlag (1982).

\bibitem{Jordan} C. Jordan,
\newblock {\it Essais sur la G\'eom\'etrie \`a $n$ Dimensions  },
\newblock Bull. Soc. Math. France {\bf 3}, 103-174 (1875).


\bibitem{Mu-Ni} M. I. Munteanu and A. I. Nistor,
\newblock {\it A new approach on constant angle surfaces in $\R^3$},
\newblock  Turkish J. Math.  {\bf 33:2}  (2009), 169-178.

\bibitem{Wei} J. L. Weiner,
\newblock {\it The Gauss map for surfaces in 4-Space},
\newblock Math. Ann. {\bf 269} (1984), 541-560.


\end{thebibliography}
\end{document}